\documentclass{amsart}
\usepackage{amsmath,amsfonts,amssymb}
\usepackage{graphicx}

\newtheorem{theorem}{Theorem}[section]

\newtheorem{definition}[theorem]{Definition}

\newtheorem{lemma}[theorem]{Lemma}

\begin{document}

\title[On the complexity function for sequences which are not uniformly recurrent]{On the complexity function for sequences which are not uniformly recurrent}

\begin{abstract}
We prove that every non-minimal transitive subshift $X$ satisfying a mild aperiodicity condition satisfies $\limsup c_n(X) - 1.5n = \infty$, and give a class of examples which shows that the threshold of $1.5n$ cannot be increased. As a corollary, we show that any transitive $X$ satisfying $\limsup c_n(X) - n = \infty$ and $\limsup c_n(X) - 1.5n < \infty$ must be minimal. We also prove some restrictions on the structure of transitive non-minimal $X$ satisfying $\liminf c_n(X) - 2n = -\infty$, which imply unique ergodicity (for a periodic measure) as a corollary, which extends a result of Boshernitzan \cite{bosh} from the minimal case to the more general transitive case. 
\end{abstract}

\date{}
\author{Nic Ormes}
\address{Nic Ormes\\
Department of Mathematics\\
University of Denver\\
2390 S. York St.\\
Denver, CO 80208}
\email{normes@du.edu}
\urladdr{www.math.du.edu/$\sim$normes/}
\author{Ronnie Pavlov}
\address{Ronnie Pavlov\\
Department of Mathematics\\
University of Denver\\
2390 S. York St.\\
Denver, CO 80208}
\email{rpavlov@du.edu}
\urladdr{www.math.du.edu/$\sim$rpavlov/}
\thanks{The second author gratefully acknowledges the support of NSF grant DMS-1500685.}
\renewcommand{\subjclassname}{MSC 2010}

\maketitle

\section{Introduction and definitions}\label{Intro}

In this work, we describe some simple connections between the recurrence properties of a two-sided sequence $x$ and the so-called \textbf{word complexity function} $c_n(x)$ which measures the number of words of length $n$ appearing in $x$. 
One of the most fundamental results of this sort is the Morse-Hedlund theorem, which has slightly different statements in the one- and two-sided cases (see \cite{paul}).

\begin{theorem}{\rm (Morse-Hedlund Theorem)}\label{MH}
Suppose that $A$ is a finite alphabet, $x \in A^{\mathbb{N}}$ or $x \in A^{\mathbb{Z}}$, and there exists $n$ such that the number of $n$-letter subwords of $x$ is less than or equal to $n$. Then, if $x \in A^{\mathbb{Z}}$, $x$ must be periodic, and if $x \in A^{\mathbb{N}}$, then $x$ must be eventually periodic.
\end{theorem}

One way to view this theorem is that it yields a lower bound on $c_n(x)$; if $x$ is two-sided and not periodic, then $c_n(x) \geq n + 1$ for all $n$. It is well-known that this bound is sharp; there exist aperiodic sequences called \textbf{Sturmian} sequences (see Chapter 6 of \cite{fogg} for an introduction) for which $c_n(x) = n + 1$ for all $n$. There are also other examples in the literature (\cite{aberkane}) with $1 < \liminf (c_n(x)/n) < \limsup (c_n(x)/n) < 1 + \epsilon$ for arbitrarily small $\epsilon$. All of these examples are \textbf{uniformly recurrent} sequences, meaning that for every subword $w$, there exists $N$ so that every $N$-letter subword contains $w$. Equivalently, a sequence is uniformly recurrent whenever the shift map acting on its orbit closure forms a minimal topological dynamical system. 

There are also fairly simple examples of sequences which are not uniformly recurrent and yet have $c_n(x) < n + k$ for all $n$ and some constant $k$, given by any $x$ which is not periodic but is eventually periodic in both directions. For example, if $x = \ldots 1212344444 \ldots$ then $c_n(x) = n+3$ for all $n \geq 1$, and $x$ is clearly not uniformly recurrent since the word $1234$ occurs just once. 

This leads to a natural question: must a sequence with complexity function ``close to $n$'' be either uniformly recurrent or eventually periodic in both directions? Our main result shows that this is indeed the case.

\begin{theorem}\label{mainthm}
If $x$ is not uniformly recurrent and it is not true that $x$ is eventually periodic in both directions, then $\limsup (c_n(x) - 1.5n) = \infty$. 
\end{theorem}

This gives a large gap in the complexity functions achievable by sequences which are not uniformly recurrent; any such complexity is either below $n + k$ for all $n$ and some constant $k$, or has a subsequence along which $c_n(x) - 1.5n$ approaches infinity. In particular, this means that some interesting examples from the literature (\cite{aberkane}, \cite{heinis}) with $1 < \limsup (c_n(x)/n) < 1.5$ can only be achieved by uniformly recurrent sequences. We also show that Theorem~\ref{mainthm} is tight in the sense that the threshold of $1.5n$ cannot be meaningfully increased.

\begin{theorem}\label{mainex}
For any nondecreasing $g: \mathbb{N} \rightarrow \mathbb{R}$ with $\lim g(n) = \infty$, there exists an $x$ which is not uniformly recurrent where $c_n(x) < 1.5n + g(n)$ for sufficiently large $n$.
\end{theorem}

Our proof is an analysis by cases, and in most of the cases, the much stronger bound $\liminf (c_n(x) - 2n) > - \infty$ holds. We can then prove a fairly strong structure on those $x$ for which it does not.

\begin{theorem}\label{structhm}
If $x$ is not uniformly recurrent, it is not true that $x$ is eventually periodic in both directions, and $\liminf (c_n(x) - 2n) = - \infty$, then there exist a constant $k$ and periodic orbit $M$ with the following property: for every $N$, there exists $m > N$ so that 
every $(3m + k)$-letter subword of $x$ contains an $m$-letter subword of a point in $M$. 
\end{theorem}

Informally, the conclusion of Theorem~\ref{structhm} says that $x$ can be partitioned, at arbitrarily large ``scales,'' into long 
(possibly infinite on one side) pieces of the periodic orbit $M$ and pieces not in $M$ which are not much longer. Unsurprisingly, this structure is quite similar to the structure of the examples proving Theorem~\ref{mainex}, as we will see in Section~\ref{proofs}.

Theorem~\ref{structhm} implies a useful corollary which extends a result of Boshernitzan. He proved in \cite{bosh} that if $X$ is minimal and $\liminf (c_n(X) - 2n) = -\infty$, then $X$ is uniquely ergodic, i.e. there is only one shift-invariant Borel probability measure on $X$. The following result uses the same complexity hypothesis, but applies to non-minimal transitive systems. 

\begin{theorem}\label{mainthm2}
If $X = \overline{\mathcal{O}(x)}$, $x$ is not uniformly recurrent, it is not true that $x$ is eventually periodic in both directions, and $\liminf (c_n(X) - 2n) = -\infty$, then $X$ is uniquely ergodic, with unique shift-invariant measure supported on a periodic orbit.
\end{theorem}
(We would like to note that the proof in \cite{bosh} could theoretically be applied to transitive systems with very few changes, and so the main new content in our result is the triviality of the measure in the non-minimal case.)

Cyr and Kra (\cite{cyr-kra}) recently generalized a different result of Boshernitzan's, proving that under no assumption on $X$ whatsoever, for any $k \in \mathbb{N}$, $\liminf (c_n(X)/n) < k$ implies that $X$ has fewer than $k$ nonatomic shift-invariant measures which have so-called generic points. 
Theorem~\ref{mainthm2} applies only to the case $k = 2$ and assumes transitivity of $X$ and some aperiodicity of $x$, but uses a weaker complexity hypothesis and implies that $X$ cannot have multiple shift-invariant measures at all, rather than only forbidding multiple nonatomic shift-invariant measures.


\section{Definitions}\label{defs}

Let $A$ denote a finite set, which we will refer to as our alphabet.

\begin{definition} 
A bi-infinite sequence $x \in A^{\mathbb{Z}}$ is \textbf{periodic} if there exists $n \neq 0$ so that $x(k) = x(k+n)$ for all $k \in \mathbb{Z}$. A one-sided sequence $x \in A^{\mathbb{N}}$ is \textbf{eventually periodic} if there exist $n, N \in \mathbb{N}$ so that 
$x(k) = x(k+n)$ for all $k > N$; the definition is analogous for $x \in A^{-\mathbb{N}}$. A bi-infinite sequence $x$ is \textbf{eventually periodic in both directions} if $x(0) x(1) x(2) \ldots$ and $\ldots x(-2) x(-1)$ are each eventually periodic.
\end{definition}

\begin{definition}
A \textbf{subshift} $X$ on an alphabet $A$ is any subset of 
$A^{\mathbb{Z}}$ which is invariant under the left shift map $\sigma$ and closed in the product topology.
\end{definition}

\begin{definition}
A subshift $X$ is \textbf{transitive} if it can be written as $\overline{\mathcal{O}(x)}$ for some $x \in A^{\mathbb{Z}}$, where $\mathcal{O}(x) := \{\sigma^n x \ : \ n \in \mathbb{Z}\}$.
\end{definition}

\begin{definition}
A subshift $X$ is \textbf{minimal} if it contains no proper nonempty subshift; equivalently, if $X = \overline{\mathcal{O}(x)}$ for all $x \in X$.
\end{definition}

A routine application of Zorn's Lemma shows that every nonempty subshift contains a nonempty minimal subshift.

\begin{definition}
A \textbf{word} over $A$ is a member of $A^n$ for some $n \in \mathbb{N}$, which we call the \textbf{length} of $w$ and denote by 
$|w|$. A word $w$ is called a \textbf{subword} of a longer word or infinite or bi-infinite sequence $u$ if there exists $i$ so that $u(i + j) = w(j)$ for all $1 \leq j \leq |w|$.
\end{definition}

\begin{definition}
A sequence $x \in A^{\mathbb{Z}}$ is \textbf{recurrent} if every subword of $x$ appears infinitely many times within $x$, and \textbf{uniformly recurrent} if, for every $w \in W(x)$, there exists $N$ so that every $N$-letter subword of $x$ contains $w$ as a subword.
\end{definition}

\begin{definition} For any words $v \in A^n$ and $w \in A^m$, we define the \textbf{concatenation} $vw$ to be the word in $A^{n+m}$ whose first $n$ letters are the letters forming $v$ and whose next $m$ letters are the letters forming $w$.
\end{definition}

\begin{definition} For a word $u \in A^n$, if $u$ can be written as the concatenation of two words $u=vw$ then we say that 
$v$ is a \textbf{prefix} of $u$ and that $w$ is a \textbf{suffix} of $u$. 
\end{definition}

\begin{definition} 
For any infinite or bi-infinite sequence $x$, we denote by $W(x)$ the set of all subwords of $x$ and, for any $n \in \mathbb{N}$, 
define $W_n(x) = W(x) \cap A^n$, the set of subwords of $x$ with length $n$. For a subshift $X$, we define $W(X) = \bigcup_{x \in X} W(x)$ and $W_n(X) = \bigcup_{x \in X} W_n(x)$. 
\end{definition}

\begin{definition} 
For any infinite or bi-infinite sequence $x$, $c_n(x) := |W_n(x)|$ is the \textbf{word complexity function} of $x$; for a subshift $X$, $c_n(X)$ is similarly defined.
\end{definition}

\begin{definition}
A word $w$ is \textbf{right-special} within a subshift $X$ if there exist $a \neq b \in A$ so that $wa, wb \in W(X)$. 
\end{definition}

We note that for every subshift $X$ and $w \in W(X)$, there exists at least one letter $a$ so that $wa \in W(X)$. Therefore, for any $n$, $c_{n+1}(X) - c_n(X)$ is greater than or equal to the number of right-special words in $W_n(X)$. 

\begin{definition}
A \textbf{sliding block code with window size $k$} is a function $\phi$ defined on a subshift $X$ where $\phi(X)$ is a subshift and $(\phi(x))(i)$ depends only on $x(i) x(i+1) \ldots x(i + k - 1)$ for all $x \in X$ and $i \in \mathbb{Z}$.
\end{definition}

For a sliding block code $\phi$ with window size $k$, even though $\phi$ technically is defined on $X$, it induces an obvious action on words in $W_n(X)$ for $n \geq k$; for any such $w$, one can define $\phi(w) \in W_{n - k + 1}(\phi(X))$ to be $(\phi(x))(0) \ldots (\phi(x))(n-k+1)$ for any $x$ with $x(0) \ldots x(n-1) = w$. (This is independent of choice of $x$ by the definition of sliding block code.) This induces a surjection from $W_n(X)$ to $W_{n-k+1}(\phi(X))$, and so for any such $\phi$ and $n \geq k$, 
$c_n(X) \geq c_{n - k + 1}(\phi(X))$.

\section{Proofs}\label{proofs}

\subsection{Proof of Theorem~\ref{mainthm}}

Throughout, $x$ will represent a bi-infinite sequence and $X$ will represent its orbit closure, $X= \overline{\mathcal{O}(x)}$. Note that then $W_n(X)$ is just the set of words of length $n$ appearing as subwords of $x$, and $c_n(X)$ is the number of such words, i.e., $W_n(X)=W_n(x)$ and $c_n(X)=c_n(x)$. We assume throughout that $x$ is not uniformly recurrent and that it is not true that $x$ is eventually periodic in both directions, and will now break into various cases and give lower bounds on $c_n(x)$ in each.

\subsubsection{$x$ is non-recurrent}\label{nonrec}



\begin{lemma}\label{nonreclem}
If $x$ is non-recurrent and it is not true that $x$ is eventually periodic in both directions, then there exists a constant $k$ so that $c_n(X) > 2n - k$ for all $n$.
\end{lemma}

\begin{proof}

Since $x$ is not recurrent, there exists a word $v$ which appears in $x$ only finitely many times. We can then write $x = \ell w r$ where $w$ contains all occurrences of $v$ in $x$; then $w$ occurs only once in $x$. Then $\ell$ and $r$ do not contain $w$, and one of $\ell$ or $r$ is not eventually periodic. We treat only the $r$ case here, as the $\ell$ case is similar. Since $r$ is not eventually periodic, by Theorem~\ref{MH}, it contains at least $n+1$ distinct $n$-letter subwords for every $n$, and none of these contain $w$ as a subword. In addition, $x = \ell w r$ contains $n-|w|+1$ subwords of length $n$ which contain $w$, which are all distinct since they contain $w$ exactly once at different locations. Therefore, $c_n(X) \geq (n + 1) + (n - |w| + 1) = 2n - |w| + 2$ for all $n$, completing the proof.
\end{proof}

\subsubsection{$x$ is recurrent and not uniformly recurrent}\label{nonunifrec}

\

\

In this case, since $x$ is not uniformly recurrent, $X$ is not minimal. Then $X$ must properly contain some minimal subshift. 

\begin{lemma}\label{minimal1}
If $X$ properly contains an infinite minimal subshift $M$, then there exists a constant $k$ so that $c_n(X) > 2n - k$ for all $n$.
\end{lemma}

\begin{proof}
Suppose that $X$, $M$ are as in the theorem. Since $\overline{\mathcal{O}(x)} = X \neq M$, $x$ contains a subword not in $W(M)$, let's call it $w$. By shifting $x$ if necessary, we may assume that $w = x(0) \ldots x(|w| - 1)$. By recurrence, $x$ contains infinitely many occurrences of $w$. However, since $\overline{\mathcal{O}(x)} = X \supset M$, $x$ contains arbitrarily long subwords in $W(M)$, none of which may contain $w$. Choose any $n \geq |w|$, and consider a subword of $x$ of length $n$ which does not contain $w$; take it to be $x(k) \ldots x(k+n-1)$, and for now assume that $k > 0$. Now, the word $x(0) \ldots x(k + n - 1)$ contains $w$ at least once (as a prefix), so we may define the rightmost occurrence of $w$ within it; say this happens at 
$x(j) \ldots x(j + |w| - 1)$. Note that since $x(k) \ldots x(k + n - 1)$ contains no occurrences of 
$w$, we know that $j < k$. 

Finally, consider the $n$-letter subwords of $x$ defined by $u_i = x(i) \ldots x(i+n-1)$, where 
$j - n + |w| \leq i \leq j$. Each contains the occurrence of $w$ at $x(j) \ldots x(j + |w| - 1)$, and no occurrence of $w$ to the right, by definition of $j$. Therefore, all are distinct, and so $x$ contains $n - |w| + 1$ subwords of length $n$, which each contain $w$. 

On the other hand, since $M$ is an infinite minimal subshift, it is aperiodic, and so by Theorem~\ref{MH}, $W_n(M)$ contains at least $n + 1$ subwords of length $n$, none of which contain $w$ since $w \notin W(M)$. Since $M \subset X$, $W_n(M) \subset W_n(X)$, and so $c_n(X) > (n + 1) + (n - |w| + 1) = 2n - |w| + 2$ for all $n \geq |w|$, completing the proof when $k > 0$. Since the complexity function is unaffected by reflecting $x$ (and $w$) about the origin, the same holds when $k < 0$, completing the proof. 

\end{proof}

We now only need treat the case where $X$ contains only finite minimal subshifts (i.e. periodic orbits), and will first deal with the case where it contains more than one.

\begin{lemma}\label{minimal2}
If $X$ contains two minimal subshifts and it is not true that $x$ is eventually periodic in both directions, then there exists a constant $k$ so that $c_n(X) > 2n - k$ for all $n$.
\end{lemma}

\begin{proof}
Denote by $M$ and $M'$ minimal subshifts of $X$; by definition of minimality, $M$ and $M'$ are disjoint. If either is infinite, then we are done by Lemma~\ref{minimal1}. So, assume that both are finite, and therefore periodic orbits. Choose $k$ so that $W_k(M)$ and $W_k(M')$ are disjoint, and their union is strictly contained in $W_k(X)$. Since $X = \overline{\mathcal{O}(x)}$, $x$ contains arbitrarily long words in $W(M)$ and $W(M')$. As above, we may assume without loss of generality that $x$ is not eventually periodic to the right. Then, for all 
$n \geq k$, there exists $\ell$ so that $x(\ell) \ldots x(\ell + n - 1) \in W(M)$ and $x(\ell + n - k + 1) \ldots x(\ell + n) \notin W_k(M)$. Similarly, there exists $m$ so that $x(m) \ldots x(m + n - 1) \in W(M')$ and $x(m + n - k + 1) \ldots x(m + n) \notin W_k(M')$.

Define the $n$-letter words $u_i = x(\ell+i) \ldots x(\ell + i + n - 1)$ and $v_j = x(m+j) \ldots x(m +j + n - 1)$ for $0 < i, j \leq n - k$; clearly all are in $W_n(X)$. In each $u_i$, the leftmost $k$-letter word not in $W_k(M)$ is $u_i(n - k - i + 2) \ldots u_i(n - i + 1) = x(m + n - k + 1) \ldots x(m + n)$, and so all $u_i$ are distinct. The same argument (using $W_k(M')$) shows that all $v_j$ are distinct. Finally, all $u_i$ begin with a word in $W_k(M)$ and all $v_j$ begin with a word in $W_k(M')$, and so the sets $\{u_i\}$ and $\{v_j\}$ are also disjoint. Therefore, $c_n(X) \geq 2n - 2k$ for $n > k$, completing the proof.

\end{proof}

The remaining case is that $x$ is recurrent and that $X$ properly contains a periodic orbit $M$, which is the only minimal subshift contained in $X$. For simplicity, we assume that $M$ is a single fixed point, which we may do via the following lemma.

\begin{lemma}\label{0reduction}
Suppose that $x$ is recurrent and $X = \overline{\mathcal{O}(x)}$ strictly contains a periodic orbit $M$ which is the only minimal subshift contained in $X$. Then there is a sliding block code $\phi$ with the following properties: $\phi(X)$ has alphabet $\{0,1\}$, $\phi(X)$ strictly contains the unique minimal subshift $\{0^{\infty}\}$, and $\phi(w) = 0^i$ implies that $w \in W(M)$.
\end{lemma}

\begin{proof}
Choose such $X$ and $M$, and choose any $k$ greater than the period $p$ of $M$. Define $\phi$ as follows: for every $i$, $(\phi(x))(i) = 0$ if $x(i) \ldots x(i + k - 1) \in W_k(M)$, and $1$ otherwise. Trivially $\phi(X)$ has alphabet $\{0,1\}$. If $\phi(w) = 0^i$, then $w$ has period $p$ (since all words in $W_k(M)$ have period $p$) and begins with a $k$-letter word in $W(M)$, and is therefore itself in $W(M)$. Since $X \supsetneq M$, $\phi(X)$ contains points other than $0^{\infty}$. Finally, if $\phi(X)$ contained a minimal subshift not equal to $\{0^{\infty}\}$, then it would be disjoint from $\{0^{\infty}\}$, and so its preimage would contain a minimal subshift of $X$ other than $M$, a contradiction. 

\end{proof}

\subsubsection{$x$ is recurrent, $x \in \{0,1\}^{\mathbb{Z}}$, $M = \{0^{\infty}\}$ is the only minimal subsystem of $X$}\label{hardcase}

\

\

In this case, $x$ must contain infinitely many $1$s (by recurrence) and must contain $0^n$ as a subword for every $n$ (since $0^{\infty} \in \overline{\mathcal{O}(x)} = X$). We will need the following slightly stronger fact.

\begin{lemma}\label{leftright}
For $x$ satisfying the conditions of this section, and for all $n$, $0^n 1$ and $1 0^n$ are subwords of $x$.
\end{lemma}

\begin{proof}
Choose any $n$. We know already that $0^n$ is a subword of $x$. If neither $0^n 1$ nor $1 0^n$ were subwords of $x$, then every occurrence of $0^n$ in $x$ would force $0$s on both sides, implying 
$x = 0^{\infty}$, a contradiction. Therefore, either $0^n 1$ or $1 0^n$ is a subword of $x$; assume without loss of generality that it is the former. Then by recurrence, $0^n 1$ appears twice as a subword of $x$, implying that $x$ contains a subword of the form $0^n 1 w 0^n 1$. Remove the terminal $1$, and consider the rightmost $1$ in the remaining word; it must be followed by $0^n$, and so $x$ also (in addition to $0^n 1$) contains $1 0^n$ as a subword. Since $n$ was arbitrary, this completes the proof. 

\end{proof}

By Lemma~\ref{leftright}, for every $n$ there exists a one-sided sequence $y_n$ beginning with $1$ so that $0^n y_n$ appears in $x$. By compactness, there exists a limit point $y$ of the $y_n$ (which begins with $1$), and then since $X$ is closed, $0^{\infty} y \in X$. Similarly, there exists a one-sided sequence $z$ ending with $1$ so that $z 0^{\infty} \in X$. We first treat the case whether either $y$ or $z$ is not unique.

\begin{theorem}\label{multyz}
For $x$ satisfying the conditions of this section, if there exist either $y \neq y' \in \{0,1\}^{\mathbb{N}}$ beginning with $1$ for which $0^{\infty} y, 0^{\infty} y' \in X$ or $z \neq z' \in \{0,1\}^{-\mathbb{N}}$ ending with $1$ for which $z 0^{\infty}, z' 0^{\infty} \in X$, then there exists $k$ so that $c_n(X) > 2n - k$ for all $n$.
\end{theorem}

\begin{proof} 
We prove only the statement for $y,y'$, as the corresponding proof for $z,z'$ is trivially similar. Assume that such $y,y'$ exist. Since $y \neq y'$, there exists $k$ so that $y(k) \neq y'(k)$.

For any $n \geq k$, define the $n$-letter words $u_i = 0^i y(1) \ldots y(n-i)$ and $v_i = 0^i y'(1) \ldots y'(n-i)$ for $0 \leq i \leq n-k$. First, note that $u_i$ and $v_i$ both begin with $0^i 1$ for every $i$, and since $0^i 1$ is never a prefix of $0^j 1$ for $i \neq j$, the sets $\{u_i, v_i\}$ and $\{u_j, v_j\}$ are disjoint whenever $i \neq j$. Finally, for every $i$, 
$u_i(i+k) = y(k) \neq y'(k) = v_i(i+k)$, so $u_i \neq v_i$.
This yields $2n-2k+2$ words in $W_n(X)$ for $n \geq k$, completing the proof.
\end{proof}

We from now on assume that $y$ and $z$ are unique sequences beginning with $1$ and ending with $1$ respectively which satisfy $0^\infty y, z 0^\infty \in X$. 

\begin{theorem}\label{termyz}
For $x$ satisfying the conditions of the section, if either $y$ or $z$ contains only finitely many $1$s, then there exists $k$ so that $c_n(X) > 2n - k$ for every $n$.
\end{theorem}

\begin{proof}
We again treat only the $y$ case, as the proof for the $z$ case is similar. Suppose that $y$ contains only finitely many $1$s. Then, we can write $y = w 0^{\infty}$ for some $w$ beginning and ending with $1$, and $0^{\infty} y = 0^{\infty} w 0^{\infty} \in X$. By recurrence,
$x$ contains a subword $v$ which contains more than $|w|$ $1$s. 

Again, by recurrence $x$ contains $v$ infinitely many times. Also, $x$ contains the subword $0^n$ for all $n$, which never contains $v$. Therefore, for every $n$, there exists a word $u$ of length $n$ so that either $vu$ or $uv$ is a subword of $x$ and contains $v$ only once as a subword. We treat only the former case, as the latter is similar, and so suppose that $x(k) \ldots x(k+n+|v|-1) = vu$. 

For any $n \geq \max(|v|, |w|)$, consider the $n$-letter subwords of $x$ given by $t_j = x(j) \ldots x(j + n - 1)$ for $k-n+|v| \leq j \leq k$. The rightmost occurrence of $v$ within $t_j$ begins at the $(k - j + 1)$th letter of $t_j$, and so all $t_j$ are distinct. This yields $n - |v| + 1$ words in $W_n(X)$ which each contain $v$. On the other hand, we can define $u_i = 0^i w 0^{n - |w| - i}$ for $0 \leq i \leq n - |w|$, each of which is contained in $0^{\infty} w 0^{\infty} \in X$. Each $u_i$ contains $w$ exactly once, beginning at the $(i+1)$th letter, and so all are distinct. In addition, each $u_i$ contains at most $|w|$ $1$s, and so none contains $v$, meaning no $t_j$ and $u_i$ can be equal.

Therefore, for $n \geq \max(|v|, |w|)$, $c_n(X) \geq (n - |v| + 1) + (n - |w| + 1) = 2n - |v| - |w| + 2$, completing the proof.

\end{proof}

\begin{theorem}\label{inf0run}
For $x$ satisfying the conditions of the section, if the lengths of runs of $0$s in $y$ or $z$ are bounded, then there exists $k$ so that $c_n(X) > 2n - k$ for all $k$. 
\end{theorem}

\begin{proof}
As usual, we treat only the $y$ case since the $z$ case is similar. Suppose that there exists $k$ so that $0^k$ is not a subword of $y$. Then, for any $n > k$, consider the $n$-letter words $u_i = 0^i y(1) \ldots y(n-i)$, $k \leq i < n$, and $v_j = z(-j) \ldots z(-1) 0^{n-j}$, $0 < j \leq n - k$. Each $u_i$ begins with $0^i 1$, and $0^i 1$ is never a prefix of $0^{i'} 1$ for $i \neq i'$, so all $u_i$ are distinct; a similar argument shows that all $v_j$ are distinct. In addition, all $v_j$ end with $0^k$, and all $u_i$ either have final $k$ letters containing $y(1) = 1$ or end with a $k$-letter subword of $y$, and in either case do not end with $0^k$. Therefore, no $u_i$ and $v_j$ can be equal, and so $c_n(X) \geq 2n - k$, completing the proof. 
\end{proof}

We finally arrive at the only case in which $\liminf (c_n(x) - 2n)$ may be $-\infty$: $y$ and $z$ contain infinitely many $1$s and arbitrarily long runs of $0$s. In this case, we instead prove the weaker bound from the conclusion of Theorem~\ref{mainthm}, and interestingly only require the stated hypotheses on $y$.

\begin{theorem}\label{lastcase}
For $x$ satisfying the conditions of the section, if $y$ contains infinitely many $1$s and contains $0^n$ as a subword for all $n$, then $\limsup (c_n(x) - 1.5n) = \infty$.
\end{theorem}

\begin{proof}

For every $k$, choose $m \geq 2k$ so that $y(1) \ldots y(m)$ ends with $1$ and contains exactly $2k$ $1$s; note that $y(1) \ldots y(m)$ does not contain $0^{m - 2k + 1}$ as a subword. Then, choose $\ell$ so that $y(1) \ldots y(\ell) 0^{m - 2k + 1}$ is a prefix of $y$ and contains $0^{m - 2k + 1}$ only at the end, i.e. $y(\ell + 1) \ldots y(\ell + m - 2k + 1)$ is the first occurrence of $0^{m - 2k + 1}$ in $y$. Clearly, $\ell \geq m$ since $y(1) \ldots y(m)$ ends with $1$ and did not contain $0^{m - 2k + 1}$. Also, by definition of $\ell$, $y(1) \ldots y(\ell) 0^{m - 2k} = y(1) \ldots y(\ell + m - 2k)$ does not contain $0^{m - 2k + 1}$.


Now, consider the $(\ell + m - 2k)$-letter words defined by $u_i = 0^i y(1) \ldots y(\ell + m - 2k - i)$, $0 \leq i < \ell + m - 2k$, and $v_j = z(-j) \ldots z(-1) 0^{\ell + m - 2k - j}$, $0 \leq j < \ell$. 
Again, since each $u_i$ begins with $0^i 1$, all $u_i$ are distinct; similarly, all $v_j$ are distinct. In addition, all $v_j$ end with $0^{m - 2k + 1}$, and all $u_i$ either have final $m - 2k + 1$ letters containing $y(1) = 1$ or end with a subword of $y(1) \ldots y(\ell + m - 2k)$, and in either case do not end with $0^{m - 2k + 1}$. Therefore, no $u_i$ and $v_j$ can be equal, and so $c_{\ell + m - 2k}(X) \geq (\ell + m - 2k) + \ell = 2\ell + m - 2k$.

Recall that $\ell \geq m$; therefore, $2\ell + m - 2k \geq 1.5\ell + 1.5m - 2k = 1.5(\ell + m - 2k) + k$. In other words, for $n = \ell + m - 2k$, $c_n(X) \geq 1.5n + k$. Since $k$ was arbitrary, $c_n(x) - 1.5n$ is unbounded from above, completing the proof.

\end{proof}

We are now prepared to combine the results from the previous subsections to prove Theorem~\ref{mainthm}.

\begin{proof}[Proof of Theorem~\ref{mainthm}]

We assume that $X$ is not minimal and that it is not the case that $x$ is eventually periodic in both directions. 
By Lemma~\ref{nonreclem}, if $x$ is non-recurrent, then $\liminf (c_n(X) - 2n) > - \infty$, implying that $\limsup (c_n(X) - 1.5n) = \infty$.
By Lemmas~\ref{minimal1} and \ref{minimal2}, if $X$ contains either two minimal subsystems or an infinite minimal subsystem, then $\liminf (c_n(X) - 2n) > - \infty$, implying that $\limsup (c_n(X) - 1.5n) = \infty$.

So, we can assume that $x$ is recurrent and that $X$ properly contains a unique minimal subsystem, which is finite. Take the sliding block code $\phi$ (with window size $k$) guaranteed by Lemma~\ref{0reduction}. If we define $y = \phi(x)$ and $Y = \phi(X)$, then $Y = \overline{\mathcal{O}(y)}$ has alphabet $\{0,1\}$, strictly contains the unique minimal subshift 
$\{0^{\infty}\}$, and (since $\phi$ has window size $k$) satisfies $c_n(X) \geq c_{n-k+1}(Y)$ for all $n$.

By Theorems~\ref{multyz}, \ref{termyz}, \ref{inf0run}, and \ref{lastcase}, $\limsup (c_n(Y) - 1.5n) = \infty$, and since $c_n(X) \geq c_{n-k+1}(Y)$ for all $n$, it must be the case that $\limsup (c_n(X) - 1.5n) = \infty$, completing the proof.

\end{proof}

\subsection{Proof of Theorem~\ref{mainex}}\label{tight}

Fix any nondecreasing unbounded $g: \mathbb{N} \rightarrow \mathbb{R}$. Clearly there exist $N$ and a nondecreasing unbounded $f: \mathbb{N} \rightarrow \mathbb{N}$ so that
$f(n) + 1 \leq g(n)$ for all $n > N$.

We will construct a point $x \in \{0,1\}^{\mathbb{Z}}$ of the following form 
$$x = 0^{\infty}.\ 1\  0^{g_1} \ 1\  0^{g_2}\  1\  0^{g_3}\  1\  0^{g_4}\ 1 \ldots$$
where all $g_i \geq 1$. We will refer to these numbers $\{g_i\}$ as the \emph{gaps} (between 
1s in $x$). Next we describe how the gaps are defined. 

We will construct an increasing sequence of natural numbers $n_1 < n_2 < n_3 < \cdots$, and for every $i$ define
$$g_i = n_k \text{ if $i$ is the product of $2^k$ and an odd natural number where $k\geq 0$.}$$
As such, $x$ will have the form 
$$x = 0^{\infty}.\ 1\ 0^{n_0}\ 1\ 0^{n_1}\ 1\ 0^{n_0}\ 1\ 0^{n_2}\ 1\ 0^{n_0}\  1\  0^{n_1}\ 1\ 0^{n_0}\ 1\ 0^{n_3}\ldots $$

Our goal is to show that 
the natural numbers $n_0< n_1 < n_2 < \cdots$ may be chosen so that
$c_n(x) < 1.5n + 1 + f(n)$ for all $n \in \mathbb{N}$; since $1.5 n + 1 + f(n) \leq 1.5n + g(n)$ for $n > N$, we will then be done.

We will establish this by consideration of right-special words occurring in $x$ of various lengths. 
First note that for any $j \geq 1$, $0^j$ is a right-special word: both $0^{j+1}$ and $0^j1$ appear in 
$x$. 

Set $w(0)=0^{n_0}10^{n_0}$, and for $k \geq 1$, let $w(k)$ be the unique word in $x$ of the form $w=0^{n_k}1u10^{n_k}$ where $u$ has no occurrence of $0^{n_k}$. The uniqueness of $w(k)$ can be seen from noting that a gap of $n_k$ or longer must correspond to the $i$th gap in $x$ where $i$ is multiple of $2^{k}$. By the construction of $x$, 
the sequence of gaps that occur between any two consecutive gaps of $n_k$ or more are always the same and are equal to $g_1, g_2, \ldots , g_{2^{k}-1}$. Thus 
$$w(k) = 0^{n_k}1 0^{g_1} 1 0^{g_2}1 \cdots 1 0^{g_{2^k-1}} 1 0^{n_k} = 0^{n_k} 1 0^{n_0} 1 0^{n_1} 1 0^{n_0} 1 0^{n_2} \cdots 0^{n_2} 1 0^{n_0} 1 0^{n_1} 1 0^{n_0} 1 0^{n_k}.$$
We make a series of claims about the words $w(k)$ for all $k \geq 0$. 

Claim 1: Every $w(k)$ is right-special. Because there are gaps of exactly $n_k$, $w(k)1$ occurs in $x$, and because there are gaps larger than $n_k$, $w(k)0$ occurs in $x$. 

Claim 2: Neither $0w(k)$ nor $1w(k)$ are right-special. 
Given two consecutive multiples of $2^k$, one is the product of an odd natural number and $2^k$ and the other is a multiple of $2^{k+1}$. Therefore, given two consecutive gaps of $n_k$ or longer, one is exactly $n_k$ and the other is strictly more than $n_k$. Therefore, neither $0w(k)0$ nor $1w(k)1$ occur in $x$, but both $0w(k)1$ and $1w(k)0$ occur in $x$. 

Claim 3: Any right-special word $w$ is a suffix of $w(k)$ for some $k>0$. Clearly, $w=0^n$ is a suffix of $w(k)$ for $k$ large enough that $n_k>n$. Now assume $w$ is a right-special word of the form $u10^n$ for some word $u$ and some 
$n \geq 0$. Then $w1$ occurs in $x$, meaning that $n = n_k$ for some $k \geq 0$. Therefore, $w=u10^{n_k}$. 
If $|w| \leq |w(k)|$ then $w$ is a suffix of $w(k)$ by the uniqueness of $w(k)$. Now assume $|w| > |w(k)|$. Then again by the uniqueness of $w(k)$, $w = vw(k)$ for some word $v$ with $|v|>1$. But since any suffix of a right-special word is right-special, this implies that either $0w(k)$ or $1w(k)$ is right-special, contradicting Claim 2. 

Next we give a recursive formula for $|w(k)|$. 
Between any two consecutive multiples of $2^k$, for $ 1 \leq j  \leq k$, 
there are $2^{j-1}$ odd multiples of $2^{k-j}$. Therefore, in $w(k)$ we have 
two runs of $n_k$ 0s, $2^{k}$ 1s, and $2^{j-1}$ gaps of $n_{k-j}$ for $0 < j \leq k$. This gives  
$$|w(k)| = 2^{k} + 2 n_k + \sum_{j=1}^{k} 2^{j-1} n_{k-j} = 2^{k} + 2 n_k + \sum_{j=0}^{k-1} 2^{k-j-1} n_{j}.$$

In order to analyze $c_n(x)$, we consider the number of right-special words of length $n$. 
For all $n \geq 1$, we have $0^n$ and for any $n \in (n_k,|w(k)|]$, we have a suffix of $w(k)$ that contains at least one 1. In what follows, we will always recursively choose the sequence $\{n_k\}$ so that $n_k > |w(k-1)|$, implying that the intervals $(n_k,|w(k)|]$ are pairwise disjoint. Therefore, for some values of $n$ we will have exactly one right-special word ($0^n$), and for $n$ which are in $(n_k,|w(k)|]$ for some $k$, we have exactly two right-special words ($0^n$ and the suffix of $w(k)$ of length $n$).

Let $R = \mathbb{N} \cap \bigcup_{k \geq 0} (n_k,|w(k)|]$. 
For $n \in R$ there are two right-special words of length $n$, 
and for $n \not \in R$ there is just one right-special word of length $n$. 
This gives us the recursion formula 
$$c_{n+1}(X) = c_{n}(X) + 1 + \left|\{n\} \cap R \right|.$$
From this, and the fact that $c_1(X) = 2$ it follows that 
\begin{equation}\label{rtspecialbd}
c_n(X) = n + 1 + \left|\{1,2,\ldots ,n-1\} \cap R \right|
\end{equation}
for all $n \geq 1$. 

It remains to show that the sequence $n_0 < n_1 < n_2 < \cdots$ can be chosen 
so that $\left|\{1,2,\ldots ,n-1\}  \cap R \right| < 0.5 n + f(n)$ for all $n \in \mathbb{N}$.
First, we choose $n_0 = 1$, meaning that $|w(0)| = 2n_0 + 1 = 3$. Then clearly 
$\left|\{1,2,\ldots ,n-1\} \cap R \right| \leq 2 < 0.5n + f(n)$ for $n \leq 3 = |w(0)|$.





\subsubsection{Choice of $n_k$, $k\geq 1$}

Suppose $n_0, n_1, \ldots, n_{k-1}$ have been chosen so that \newline 
$\left|\{1,2,\ldots ,n-1\}  \cap R \right| < 0.5n + f(n)$ 
for all $n \leq |w(k-1)|$. Choose 
$n_k$ so that $f(n_k)$ is greater than

\begin{equation}\label{bigk}
0.5|w(k)| - n_k + |R \cap \{1, \ldots, |w(k-1)|\}| = 2^{k-1} + \sum_{j=0}^{k-1} 2^{k-2-j} n_j + |R \cap \{1, \ldots, |w(k-1)|\}|.
\end{equation}

For $n \in (|w(k-1)|,n_k)$, we have 
\begin{align*} \left|\{1,2,\ldots ,n-1\}  \cap R \right| & = 
\left|\{1,2,\ldots ,|w(k-1)|\}  \cap R \right| \\
& <  0.5(|w(k-1)| + 1) + f(|w(k-1)| + 1) \leq 0.5n + f(n).
\end{align*}
For $n \in [n_k,|w(k)|)$ we have 
\begin{align*}
\left|\{1,2,\ldots ,n-1\}  \cap R \right| 
& = n - n_k + \left|\{1,2,\ldots ,|w(k-1)|\}  \cap R \right|\\
& = 0.5 n + \left(0.5 n - n_k + \left|\{1,2,\ldots ,|w(k-1)|\}  \cap R \right|\right)\\
& < 0.5 n + \left(0.5 |w(k)| - n_k + \left|\{1,2,\ldots ,|w(k-1)|\} \cap R \right|\right)\\
& < 0.5 n + f(n_k) \leq 0.5 n + f(n). 
\end{align*}
(The second-to-last inequality came from (\ref{bigk}).) We've shown that $\left|\{1,2,\ldots ,n-1\}  \cap R \right| < 0.5n + f(n)$ for all $n$, and so by (\ref{rtspecialbd}), $c_n(X) < 1.5n + f(n) + 1$ for all $n$.
Since $f(n) + 1 \leq g(n)$ for $n > N$, this means that $c_n(X) < 1.5n + g(n)$ for $n > N$, completing the proof.

\subsection{Proofs of Theorems~\ref{structhm} and \ref{mainthm2}}\label{structure}


\begin{proof}[Proof of Theorem~\ref{structhm}]

We first note that by Lemmas~\ref{nonreclem}, \ref{minimal1}, and \ref{minimal2}, $\liminf (c_n(X) - 2n) = -\infty$ implies that $x$ is recurrent and that $X$ properly contains a periodic orbit $M$, which is the unique minimal subshift contained in $X$. We will for now assume that $M = \{0^{\infty}\}$ and that $X$ has alphabet $\{0,1\}$, and will then extend to the general case by Lemma~\ref{0reduction}.

By Lemma~\ref{leftright}, $0^n 1 \in W(X)$ for all $n$, and so $0^n$ is right-special for all $n$. Therefore, for any $n$ where there is another right-special word in $W_n(X)$, $c_{n+1}(X) - c_n(X) \geq 2$.

For any $N$, choose $n$ so that $c_n(X) - 2n \leq -N$, and define 
$S = \{j < n \ : \ 0^j$ is the only right-special word in $W_j(X)\}$. Then $\displaystyle c_n(X) = \sum_{j = 0}^{n-1} (c_{j+1}(X) - c_j(X)) \geq |S| + 2(n - |S|)$, and so $|S| \geq N$. Define $m$ to be the maximal element of $S$; then $0^m$ is the only right-special word in $W_m(X)$, $m \geq N$, and $\displaystyle c_m(X) = c_n(X) - \sum_{j = m}^{n-1} (c_{j+1}(X) - c_j(X)) \leq 
2n - 2(n-m) = 2m$.

We claim that every word in $W_{3m}(X)$ contains $0^m$, and so that every subword of length $3m$ of $x$ contains $0^m$. Suppose for a contradiction that this is false, i.e. that there is $y \in X$ where $y(1) \ldots y(3m)$ does not contain $0^m$. Since $c_m(X) \leq 2m$, there exist $1 \leq i < j \leq 2m$ so that $y(i) \ldots y(i + m - 1) = y(j) \ldots y(j + m - 1)$. 
Also, all $m$-letter words $y(k) \ldots y(k + m - 1)$ for $i \leq k \leq j$ are not $0^m$ and so, since $m \in S$, are not right-special, i.e. there is only one letter that can follow each of them in a point of $X$. This means that $y(i) y(i+1) \ldots$ is in fact periodic with period $j - i$. Since 
$y(i) \ldots y(j + m - 1)$ does not contain $0^m$ (as a subword of $y(1) \ldots y(3m)$), 
$y(i) y(i+1) \ldots$ cannot contain $0^m$, a contradiction to $0^\infty$ being the only minimal subsystem of $X$. This means that the original claim was true, completing the proof in the case $M = \{0^{\infty}\}$.


Now suppose that $M$ is an arbitrary periodic orbit, and take the sliding block code $\phi$ (with window size $k$) guaranteed by Lemma~\ref{0reduction}. As before, define $y = \phi(x)$ and $Y = \phi(X)$; then $Y$ has alphabet $\{0,1\}$, strictly contains the unique minimal subshift $\{0^{\infty}\}$, and satisfies $c_n(X) \geq c_{n-k+1}(Y)$ for all $n$, implying that $\liminf (c_n(Y) - 2n) = -\infty$.

From the above proof, for all $N$, there exists $m \geq N$ so that every $3m$-letter subword of $y$ contains $0^m$. Every $(3m+3k)$-letter subword of $x$ has image under $\phi$ which is a $(3m+2k)$-letter subword of $y$, and therefore contains $0^m$. By definition of $\phi$, $x$ contains an $(m+k)$-letter word at the corresponding location which is in $W(M)$; since $m + k \geq m \geq N$, the proof is complete. 


\end{proof}

The proof of Theorem~\ref{mainthm2} uses a few basic notions from ergodic theory, which we briefly and informally summarize here. Firstly, a (shift-invariant Borel probability) measure $\mu$ is called \textbf{ergodic} if every measurable set $A$ with $A = \sigma A$ has $\mu(A) \in \{0,1\}$. Ergodic measures are valuable because of the \textbf{pointwise ergodic theorem}, which says that $\mu$-almost every point $x$ in $X$ is \textbf{generic for} $\mu$, which means that for every $f \in C(X)$, $\displaystyle \frac{1}{n} \sum_{i=0}^{n-1} f(\sigma^i x) \rightarrow \int f \ d\mu$. 
For the purposes of the proof below, we need only a very simple application of the ergodic theorem: for any generic point for $\mu$, the frequency of $0$ symbols is equal to $\mu([0])$, where $[0]$ is the set of $z \in X$ containing a $0$ at the origin. Finally, the \textbf{ergodic decomposition theorem} states that every (shift-invariant Borel probability) measure is a sort of generalized convex combination of ergodic measures. Again, we need only a very simple corollary: if a subshift has only one ergodic measure, then it has only one (shift-invariant Borel probability) measure. For a more detailed introduction to ergodic theory, see \cite{walters}.

\begin{proof}[Proof of Theorem~\ref{mainthm2}]


Assume that $x$ is not uniformly recurrent, it is not the case that $x$ is periodic in both directions, and $\liminf (c_n(X) - 2n) = -\infty$. Then as above, $X$ properly contains a periodic orbit $M$, which is the unique minimal subshift contained in $X$. We again first treat the case where $M = \{0^{\infty}\}$. Assume for a contradiction that $X$ has a (shift-invariant Borel probability) measure $\mu$ not equal to $\delta_{0^{\infty}}$. Since $\delta_{0^{\infty}}$ is obviously ergodic, by ergodic decomposition we may assume without loss of generality that $\mu$ is ergodic.

By ergodicity, $\mu(0^{\infty}) = 0$, and so there exists $j$ so that $\mu([0^j]) < 1/6$. Since $\mu$ is ergodic, by the pointwise ergodic theorem there is a  point $z$ which is generic for $\mu$. Note that the only periodic orbit in $X$ is $\{0^{\infty}\}$, and so $z$ cannot be eventually periodic in both directions, since then it would be generic for $\delta_{0^\infty}$. In addition, note that $\overline{\mathcal{O}(z)}$ must contain $M = \{0^{\infty}\}$ (since $M$ is the unique minimal subshift contained in $X$), and so $z$ is not uniformly recurrent.

Finally, since $\liminf (c_n(X) - 2n) = -\infty$, we know that $\liminf (c_n(z) - 2n) = -\infty$, and so $z$ satisfies the hypotheses of Theorem~\ref{structhm}. We apply that theorem with $N = 2j$ to find $m \geq 2j$ for which every $3m$-letter subword of $z$ contains $0^m$. Then, the frequency of occurrences of $0^j$ in $z$ is at least $\frac{m-j}{3m}$, which is greater than or equal to $1/6$ since $m \geq 2j$. By genericity of $z$ for $\mu$, $\mu([0^j]) \geq 1/6$, contradicting the definition of $j$ and so the existence of $\mu$.

Now, suppose that $M$ is an arbitrary periodic orbit, define the sliding block code $\phi$ (with window size $k$) guaranteed by Lemma~\ref{0reduction}, and again define $y = \phi(x)$ and $Y = \phi(X)$. As usual, $Y$ has alphabet $\{0,1\}$, strictly contains the unique minimal subshift $\{0^{\infty}\}$, and satisfies $c_n(X) \geq c_{n-k+1}(Y)$ for all $n$. We note that $y$ cannot be eventually periodic in both directions; if it were, then it would have to begin and end with infinitely many $0$s, which would imply that $x$ was eventually periodic in both directions, a contradiction. 

Finally, since $c_n(X) \geq c_{n-k+1}(Y)$ for all $n$, $\liminf (c_n(Y) - 2n) = -\infty$. So, by the proof above in the $M = \{0^{\infty}\}$ case, $Y$ has unique invariant measure $\delta_{0^{\infty}}$. Any invariant measure $\nu$ in $X$ then must have pushforward $\delta_{0^{\infty}}$ under $\phi$, and so must have $\nu(M) = 1$. It is easily checked that there is only one such $\nu$, namely the measure equidistributed over the points of $M$.

\end{proof}


\bibliographystyle{plain}
\bibliography{1.5n}

\begin{thebibliography}{1}

\bibitem{aberkane}
Ali Aberkane.
\newblock Exemples de suites de complexit\'e inf\'erieure \`a {$2n$}.
\newblock {\em Bull. Belg. Math. Soc. Simon Stevin}, 8(2):161--180, 2001.
\newblock Journ\'ees Montoises d'Informatique Th\'eorique (Marne-la-Vall\'ee,
  2000).

\bibitem{bosh}
Michael Boshernitzan.
\newblock A unique ergodicity of minimal symbolic flows with linear block
  growth.
\newblock {\em J. Analyse Math.}, 44:77--96, 1984/85.

\bibitem{cyr-kra}
Van Cyr and Bryna Kra.
\newblock Counting generic measures for a subshift of linear growth.
\newblock {\em J. Eur. Math. Soc.}, to appear.

\bibitem{fogg}
N.~Pytheas Fogg.
\newblock {\em Substitutions in dynamics, arithmetics and combinatorics},
  volume 1794 of {\em Lecture Notes in Mathematics}.
\newblock Springer-Verlag, Berlin, 2002.
\newblock Edited by V. Berth\'e, S. Ferenczi, C. Mauduit and A. Siegel.

\bibitem{heinis}
Alex Heinis.
\newblock The {$P(n)/n$}-function for bi-infinite words.
\newblock {\em Theoret. Comput. Sci.}, 273(1-2):35--46, 2002.
\newblock WORDS (Rouen, 1999).

\bibitem{paul}
Michael~E. Paul.
\newblock Minimal symbolic flows having minimal block growth.
\newblock {\em Math. Systems Theory}, 8(4):309--315, 1974/75.

\bibitem{walters}
Peter Walters.
\newblock {\em An introduction to ergodic theory}, volume~79 of {\em Graduate
  Texts in Mathematics}.
\newblock Springer-Verlag, New York-Berlin, 1982.

\end{thebibliography}

\end{document}